\documentclass[a4paper,11pt]{article}
\usepackage[utf8]{inputenc}
\usepackage{multirow}
\usepackage{amsfonts,amsmath,amssymb}
\usepackage{amsthm}
\usepackage{geometry}
\usepackage{graphicx}
\usepackage{mathrsfs}
\usepackage{tikz}
\usetikzlibrary{math}
\usepackage{enumitem}
\usepackage{multicol}
\usepackage{algorithm}
\usepackage[noend]{algpseudocode}
 
\newtheorem{theorem}{Theorem}[section]

\newtheorem{corollary}{Corollary}[section]
\newtheorem{conjecture}{Conjecture}[section]

\title{On the locating chromatic number of infinite trees}

\author
{Yusuf Hafidh $^{1}$, Devi Imulia Dian Primaskun $^{1}$, Edy Tri Baskoro\footnote{corresponding author} $^{1,2}$\\
	\\
	\normalsize{$^{1}$ \ Combinatorial Mathematics Research Group, Faculty of Mathematics }\\
        \normalsize and Natural Sciences, Institut Teknologi Bandung, Indonesia
\\
        \normalsize{$^{2}$ \ Center for Research Collaboration on Graph Theory and Combinatorics, Indonesia}\\
        \\
	\normalsize{Emails: yusuf.hafidh@itb.ac.id, imuliadevi@gmail.com,  ebaskoro@itb.ac.id}\\
}

\date{}

\begin{document} 
	
	\baselineskip18pt
	
	\maketitle 
	
\begin{abstract}
	The locating chromatic number of a graph is the smallest integer $n$ such that there is a proper $n$-coloring $c$ and every vertex has a unique vector of distances to colors in $c$. We explore the necessary conditions and provide sufficient conditions for an infinite tree to have a finite locating chromatic number. We also give an algorithm for computing the locating coloring of trees that works for both finite and infinite trees. 
\end{abstract}



\section{Introduction}
A graph is an {\em infinite graph} if it has an infinite number of vertices.
Let $G=(V, E)$ be a simple graph. For an integer $k\geq 1$, a map $c:V(G)\to \{1,2, \cdots k\}$ is called a \textit{k-coloring} of $G$ if $c(u)\neq c(v)$ for any two adjacent vertices $u,v \in V(G)$. The \textit{chromatic number} of $G$, denoted by $\chi(G)$, is the smallest positive integer $k$ such that $G$ has a \textit{k-coloring}. Let $\{X_1,X_2, \cdots ,X_k\}$ be a partition of $V$ induced by a $k$-coloring $c$. The {\em color code} $a_c(v)$ of a vertex $v$ is the ordered $k-$tuple $(d(v,X_1), d(v,X_2), \cdots, d(v,X_k))$, where $d(v,X_i) = \min\{d(v,x)|x \in X_i\}$.
If all vertices of $G$ have different color codes, then $c$ is called a {\em locating $k$-coloring} of $G$.
The {\em locating-chromatic number} of $G$, denoted by $\chi_L(G)$, is the smallest positive integer $k$ such that $G$ has a locating $k$-coloring. If there are no integer $k$ such that $G$ has a locating $k$-coloring, then $\chi_L(G)=\infty$.

A graph $G$ has {\em finite degree} if the degree of every vertex in $G$ is finite. We say that the graph $G$ has {\em bounded degree} if there is an integer $M$ such that $deg(v)\leq M$ for all $v\in V(G)$, this is equivalent with $\Delta(G)$ exists. Note that bounded degree implies finite degree but not otherwise, see Figure \ref{bounddeg}.

\begin{figure}
    \centering
    \begin{tikzpicture}[scale=0.7]
        \draw[color=blue,fill=yellow,thick] (0,0)--(10,0);
        \draw[color=blue,fill=yellow,thick] (2,1)--(2,0);
        \draw[color=blue,fill=yellow,thick] (2.8,1)--(3,0) (3.2,1)--(3,0);
        \draw[color=blue,fill=yellow,thick] (4,1)--(4.4,0) (4.4,1)--(4.4,0) (4.8,1)circle(0.12)--(4.4,0);
        \draw[color=blue,fill=yellow,thick] (5.6,1)--(6.2,0) (6,1)--(6.2,0) (6.4,1)--(6.2,0) (6.8,1)--(6.2,0);
        \draw[color=blue,fill=yellow,thick] (7.6,1)--(8.4,0) (8,1)--(8.4,0) (8.4,1)--(8.4,0) (8.8,1)--(8.4,0) (9.2,1)--(8.4,0);
        \draw[color=blue,fill=yellow,thick] (2,1)circle(0.12);
        \draw[color=blue,fill=yellow,thick] (2.8,1)circle(0.12) (3.2,1)circle(0.12);
        \draw[color=blue,fill=yellow,thick] (4,1)circle(0.12) (4.4,1)circle(0.12) (4.8,1)circle(0.12);
        \draw[color=blue,fill=yellow,thick] (5.6,1)circle(0.12) (6,1)circle(0.12) (6.4,1)circle(0.12) (6.8,1)circle(0.12);
        \draw[color=blue,fill=yellow,thick] (7.6,1)circle(0.12) (8,1)circle(0.12) (8.4,1)circle(0.12) (8.8,1)circle(0.12) (9.2,1)circle(0.12);
        \draw[color=blue,fill=yellow,thick] (0,0)circle(0.12)node[below]{$v_1$} (1,0)circle(0.12)node[below]{$v_2$} (2,0)circle(0.12)node[below]{$v_3$} (3,0)circle(0.12)node[below]{$v_4$}  (4.4,0)circle(0.12)node[below]{$v_5$}  (6.2,0)circle(0.12)node[below]{$v_6$}  (8.4,0)circle(0.12)node[below]{$v_7$};
        \foreach \x in {10.7,11,11.3}
        \draw[color=blue,fill=blue] (\x,0)circle(0.05);
    \end{tikzpicture}
    \caption{A graph with finite and unbounded degree.}
    \label{bounddeg}
\end{figure}
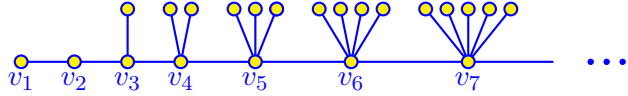

A finite path $P_n$ on $n$ vertices is usually defined for $n\geq2$. In this paper, we allow $n=1$ and hence $P_1=K_1$. Adopting the notation in \cite{DAM12}, $P_{1\infty}$ and $P_{2\infty}$ is an infinite path in one end and both ends, respectively. The locating chromatic number of a path is given in the following theorem.
\begin{theorem}\label{path}
    Let $P$ be a path (finite or infinite), then 
    \[
    \chi_L(P)=\begin{cases}
    1, & P=P_1,\\
    2, & P=P_2,\\
    3, & P=P_{1\infty}, P_{2\infty}, P_n, n\geq 3.
    \end{cases}
    \]
\end{theorem}

J. C\'aceres et al. \cite{DAM12} showed that infinite trees have finite metric dimension if and only if its degree is bounded and has a finite number of branches. However, there are infinite trees with infinite branches that have finite locating chromatic number. One example is the comb graph as depicted in Figure \ref{comb}.
\begin{figure}
    \centering
    \begin{tikzpicture}[scale=0.7]
        \draw[color=blue,thick] (-7,1)--(7,1) (0,0)node[below]{$4$}--(0,1)node[above]{$3$};
        \foreach \x in {1,3,5,-2,-4,-6} \draw[color=blue,thick] (\x,0)node[below]{$4$}--(\x,1)node[above]{$1$};
        \foreach \x in {-1,-3,-5,2,4,6} \draw[color=blue,thick] (\x,0)node[below]{$4$}--(\x,1)node[above]{$2$};
        \foreach \x in {-6,...,6} \draw[color=blue,fill=yellow,thick] (\x,0)circle(0.12) (\x,1)circle(0.12);
        \foreach \x in {7.7,8,8.3}
        \draw[color=blue,fill=blue] (\x,1)circle(0.05) (-\x,1)circle(0.05);
    \end{tikzpicture}
    \caption{A locating coloring of the comb graph.}
    \label{comb}
\end{figure}
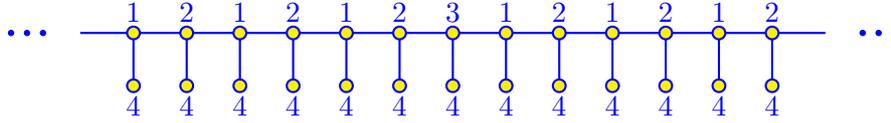
Having a bounded degree is a requirement for an infinite graph to have finite locating chromatic number. The relation between the locating chromatic number of a finite graph with its maximum degree is given in the following theorem.
\begin{theorem} (\cite{YH21})
\label{Delta xl}
If $G$ is a graph with $\chi_L(G)=k\geq3$, then $\Delta(G)\leq 4\cdot3^{k-3}$.
\end{theorem}
By generalizing the proof of Theorem \ref{Delta xl} in \cite{YH21}, we could see that the theorem still holds for infinite graphs, and hence we have the following theorem.
\begin{theorem}
    If $G$ is an infinite graph with finite locating chromatic number, then $G$ has bounded degree.
\end{theorem}

In a finite tree $T$, an {\em end-path} is a path connecting a leaf in $T$ to its nearest branch vertex. We generalize the definition of an end-path to also include an infinite path starting from a branch vertex with all vertices other than the branch vertex having degree two in $T$. If there is an end-path from a branch, we call it an {\em end-branch}. A {\em palm} is an end-branch together with all its end-paths.

\section{Infinite trees with finite locating chromatic number}

We begin this section by giving a conjecture on the characterization of infinite trees with finite locating chromatic number. We prove that one implication in this conjecture is true and give examples to support the other implication.
We define an operation $f$ on a tree $T$, namely $f(T)$ is the graph obtained by removing all end-paths of $T$ excluding the end-branches, and for $n\geq 2$, $f^n=f^{n-1}\circ f$.

\begin{conjecture} \label{TreeChar}
\label{Conj1}
    Let $T$ be an infinite tree with bounded degree. Then, $\chi_L(T) < \infty$ if and only if there is an integer $n$ such that $f^n(T)=P$ where $P$ is a path (including an infinite path, a finite path, or $K_1$).
\end{conjecture}

We prove the $(\Leftarrow)$ part of Conjecture \ref{Conj1} by repeatedly applying the following theorem.

\begin{theorem} \label{<-}
    Let $T$ be an infinite tree with bounded degree. Suppose that $\chi_L(f(T))$ is finite, then $\chi_L(T)$ is finite.
\end{theorem}
\begin{proof}
    Let $\chi_L(f(T))=m$ for some positive integer $m$ and let $c$ be a locating $m$-coloring of $f(T)$. We will construct $c'$, a locating $(m+\Delta)$-coloring of $T$. Consider the following coloring algorithm.
    
    \begin{enumerate}
        \item For every end-branch of $T$, assign an ordering of all end-paths starting from the end-branch. 
        \item For every vertex $v$ of $T$;
    \begin{enumerate}
        \item If $v$ is in $f(T)$, then $c'(v)=c(v)$; 
        \item If $v$ is not in $f(T)$, then let $b$ be the nearest end-branch from $v$;
        \item If there are more than one end-paths from $b$, suppose that $v$ is located in the $i^\text{th}$ end-path from $b$, then
        \begin{align*}
            c'(v)=\begin{cases}
            c(b), & d(b,v) \text{ is even,}\\
            m+i, & d(b,v) \text{ is odd.}
            \end{cases}.
        \end{align*}
        \item If there is only one end-path from $b$, then
        \begin{align*}
            c'(v)=\begin{cases}
            m+2, & d(b,v) \text{ is even,}\\
            m+1, & d(b,v) \text{ is odd.}
            \end{cases}.
        \end{align*}
    \end{enumerate}
    \end{enumerate}
    
    Since $T$ has bounded degree, all the vertices are colored by no more than $m+\Delta$ colors. Now, we prove that $c'$ is a locating coloring of $T$. Consider two vertices having the same distances to the colors $m+1$, $m+2,\cdots$, and $m+\Delta$. 
These two vertices must be in one of these two conditions: (i) both in $f(T)$ or (ii) both in the $i^\text{th}$ end-paths from different end-branches with the same distance to the end-branch. If these two vertices satisfy the first condition, then they will be distinguished by the color class that distinguishes them in $c$. If the second condition is satisfied then they are distinguished by the color class that distinguishes their nearest end-branches.
\end{proof}


Related to Conjecture $\ref{TreeChar}$, the contrapositive of the $(\Rightarrow)$ part states that \begin{align}\label{->}
    \text{If $f^n(T)\ne P$ for all $n$, then $\chi_L(T)=\infty$.}
\end{align}
We will provide some examples that support this statement. 
We consider two conditions for $f^n(T)\ne P$ for all $n$,
\begin{enumerate}[label=(\roman*)]
    \item $T$ has no end-path (graph $T_k$ in Figure \ref{Tk}); or
    \item $f^k(T)=T$ for some $k$ (graph $G_n$ in Figure \ref{Gk}).
\end{enumerate}

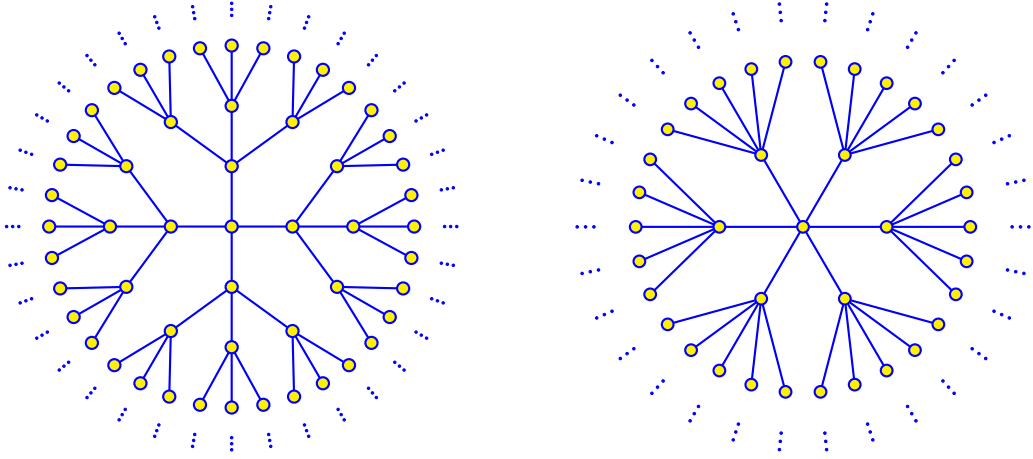
\begin{figure}
    \centering
    \begin{tikzpicture}[scale=0.8]
    \def\k {4} 
    \def\r {3}
    \def\R {0.1}

    \foreach \a in {1,...,\r}{
        \tikzmath{\b=\k*(\k-1)^(\a-1);}
        \foreach \t in {1,...,\b}{
            \tikzmath{\d=\t*360/\b;\c=round(\t/(\k-1))*(\k-1)*360/\b;}
            \draw[color=blue,fill=yellow,thick] ({\a*cos(\d)},{\a*sin(\d)})-- ({(\a-1)*cos(\c)},{(\a-1)*sin(\c)});
        }
    }
    \draw[color=blue,fill=yellow,thick] (0,0)circle(\R);
    \foreach \a in {1,...,\r}{
        \tikzmath{\b=\k*(\k-1)^(\a-1);}p
        \foreach \t in {1,...,\b}{
            \tikzmath{\d=\t*360/\b;\c=round(\t/(\k-1))*(\k-1)*360/\b;}
            \draw[color=blue,fill=yellow,thick] ({\a*cos(\d)},{\a*sin(\d)})circle(\R);
        }
    }
    \tikzmath{\b=\k*(\k-1)^(\r-1);}
    \foreach \t in {1,...,\b}{
        \tikzmath{\d=\t*360/\b;}
        \foreach \a in {0.5,0.6,0.7}{
            \draw[fill=blue,color=blue] ({(\r+\a)*cos(\d)},{(\r+\a)*sin(\d)})circle(0.02);
            }
        }
\end{tikzpicture}\qquad\qquad
\begin{tikzpicture}[scale=1.1]
    \def\k {6} 
    \def\r {2}
    \def\R {0.07}

    \foreach \a in {1,...,\r}{
        \tikzmath{\b=\k*(\k-1)^(\a-1);}
        \foreach \t in {1,...,\b}{
            \tikzmath{\d=\t*360/\b;\c=round(\t/(\k-1))*(\k-1)*360/\b;}
            \draw[color=blue,fill=yellow,thick] ({\a*cos(\d)},{\a*sin(\d)})-- ({(\a-1)*cos(\c)},{(\a-1)*sin(\c)});
        }
    }
    \draw[color=blue,fill=yellow,thick] (0,0)circle(\R);
    \foreach \a in {1,...,\r}{
        \tikzmath{\b=\k*(\k-1)^(\a-1);}
        \foreach \t in {1,...,\b}{
            \tikzmath{\d=\t*360/\b;\c=round(\t/(\k-1))*(\k-1)*360/\b;}
            \draw[color=blue,fill=yellow,thick] ({\a*cos(\d)},{\a*sin(\d)})circle(\R);
        }
    }
    \tikzmath{\b=\k*(\k-1)^(\r-1);}
    \foreach \t in {1,...,\b}{
        \tikzmath{\d=\t*360/\b;}
        \foreach \a in {0.5,0.6,0.7}{
            \draw[color=blue,fill=blue] ({(\r+\a)*cos(\d)},{(\r+\a)*sin(\d)})circle(0.015);
            }
        }
\end{tikzpicture}
    \caption{Regular infinite graphs $T_4$ and $T_6$.}
    \label{Tk}
\end{figure}

For $k\geq2$, let $T_k$ be the infinite $k$-regular tree, see Figure \ref{Tk}.

\begin{theorem}
For $k\geq 3$, $\chi_L(T_k)=\infty$.
\end{theorem}

\begin{proof}
By a contradiction, suppose there was an integer $k\geq 3$ with $\chi_L(T_k)=t$. 
Let $c$ be a locating $t$-coloring of $T_k$ and let $v \in T$ with $a_c(v)=(a_1,a_2,\cdots,a_t)$. Let $D_n(v)$ be the set of all vertices at distance $n$ from $v$, then $|D_n(v)|=k(k-1)^{n-1}$. Since $n^\alpha=o\left(\beta^n\right)$ for every real number $\alpha,\beta>1$, there is an integer $N$ such that for every $n\geq N$, $(2n)^t<k(k-1)^{n-1}$.

Let $n=\max(a_1,a_2,\cdots,a_t,N)$. Note that the distance from a vertex $u\in D_n(v)$ to any color $i$ is no more than $2a_i \leq 2n$, and hence a vertex in $D_n(v)$ has at most $(2n)^t$ possible color codes. Since $D_n(T)$ has $k(k-1)^{n-1}$ vertices, then there are at least two vertices with the same color code, a contradiction.
\end{proof}

Let $T(n,k)$ be the complete $k$-level $n$-ary tree (\cite{welary}) and $P_{1\infty}$ be an infinite path in one end with vertices $\{v_0,v_1,v_2,\dots\}$. Let $G_n$ be the graph obtained from $P_{1\infty}$ by attaching the center of $T(n,i)$ to $v_i$, see Figure \ref{Gk}.

\begin{figure}
    \begin{center}
    \def\xs {0.25}
    \def\ys {0.6}
    \begin{tikzpicture}[xscale=\xs,yscale=\ys] 
		\def\n {2}
		\def\d {5}
		\def\i {5}
		\def\rr {0.4}
		\draw[color=blue,thick,fill=red] ({\d*\i+\n^\i},0)--(0,0);
		\draw[color=blue,thick,fill=red] (0,0)circle({\rr*\ys} and {\rr*\xs}) (0,1.5)node{\Large $G_{\n}$};
		\foreach \y in {1,1.5,2} \draw[fill=blue,color=blue] ({\d*\i+\n^\i+\y},0)circle({\rr*\ys/2} and {\rr*\xs/2});
		\foreach \k in {1,...,\i}{
		\foreach \y in {1,...,\k}{
			\tikzmath{\a=\n^(\k-\y+1);\b=\n^(\y-1);}
			\foreach \x in {1,...,\a}{
			    \tikzmath{\m=mod(\x-1,\n)+1;}
			    \ifodd \k
			    \draw[thick,color=blue] (\d*\k+\b*\x+\n*\b*0.5-\m*\b,\y-\k)--(\d*\k+\b*\x-\b*0.5,\y-\k-1);
			    \else 
			    \draw[thick,color=blue] (\d*\k+\b*\x+\n*\b*0.5-\m*\b,-\y+\k)--(\d*\k+\b*\x-\b*0.5,-\y+\k+1);
			    \fi
	        }
	    }
		\tikzmath{\b=\n^\k/2;}
		\draw[color=blue,thick,fill=red] (\d*\k+\b,0)circle({\rr*\ys} and {\rr*\xs});
	    \foreach \y in {1,...,\k}{
			\tikzmath{\a=\n^(\k-\y+1);\b=\n^(\y-1);}
			\foreach \x in {1,...,\a}{
			    \ifodd \k
			    \draw[thick,color=blue,fill=red] (\d*\k+\b*\x-\b*0.5,\y-\k-1)circle({\rr*\ys} and {\rr*\xs});
			    \else
			    \draw[thick,color=blue,fill=red] (\d*\k+\b*\x-\b*0.5,-\y+\k+1)circle({\rr*\ys} and {\rr*\xs});
			    \fi
	        }
	    }
	    }
\end{tikzpicture}
\end{center}
\begin{center}
    \def\xs {0.3}
    \def\ys {0.6}
    \begin{tikzpicture}[xscale=\xs,yscale=\ys] 
		\def\n {3}
		\def\d {2}
		\def\i {3}
		\def\rr {0.4}
		\draw[color=blue,thick,fill=red] ({\d*\i+\n^\i},0)--(0,0); \draw[color=blue,thick,fill=red] (0,0)circle({\rr*\ys} and {\rr*\xs}) (0,1.5)node{\Large $G_{\n}$};
		\foreach \y in {1,1.5,2} \draw[fill=blue,color=blue] ({\d*\i+\n^\i+\y},0)circle({\rr*\ys/2} and {\rr*\xs/2});
		\foreach \k in {1,...,\i}{
		\foreach \y in {1,...,\k}{
			\tikzmath{\a=\n^(\k-\y+1);\b=\n^(\y-1);}
			\foreach \x in {1,...,\a}{
			    \tikzmath{\m=mod(\x-1,\n)+1;}
			    \ifodd \k
			    \draw[thick,color=blue] (\d*\k+\b*\x+\n*\b*0.5-\m*\b,\y-\k)--(\d*\k+\b*\x-\b*0.5,\y-\k-1);
			    \else 
			    \draw[thick,color=blue] (\d*\k+\b*\x+\n*\b*0.5-\m*\b,-\y+\k)--(\d*\k+\b*\x-\b*0.5,-\y+\k+1);
			    \fi
	        }
	    }
		\tikzmath{\b=\n^\k/2;}
		\draw[color=blue,thick,fill=red] (\d*\k+\b,0)circle({\rr*\ys} and {\rr*\xs});
	    \foreach \y in {1,...,\k}{
			\tikzmath{\a=\n^(\k-\y+1);\b=\n^(\y-1);}
			\foreach \x in {1,...,\a}{
			    \ifodd \k
			    \draw[thick,color=blue,fill=red] (\d*\k+\b*\x-\b*0.5,\y-\k-1)circle({\rr*\ys} and {\rr*\xs});
			    \else
			    \draw[thick,color=blue,fill=red] (\d*\k+\b*\x-\b*0.5,-\y+\k+1)circle({\rr*\ys} and {\rr*\xs});
			    \fi
	        }
	    }
	    }
\end{tikzpicture}
\end{center}
    \caption{Graphs $G_2$ and $G_3$}
    \label{Gk}
\end{figure}
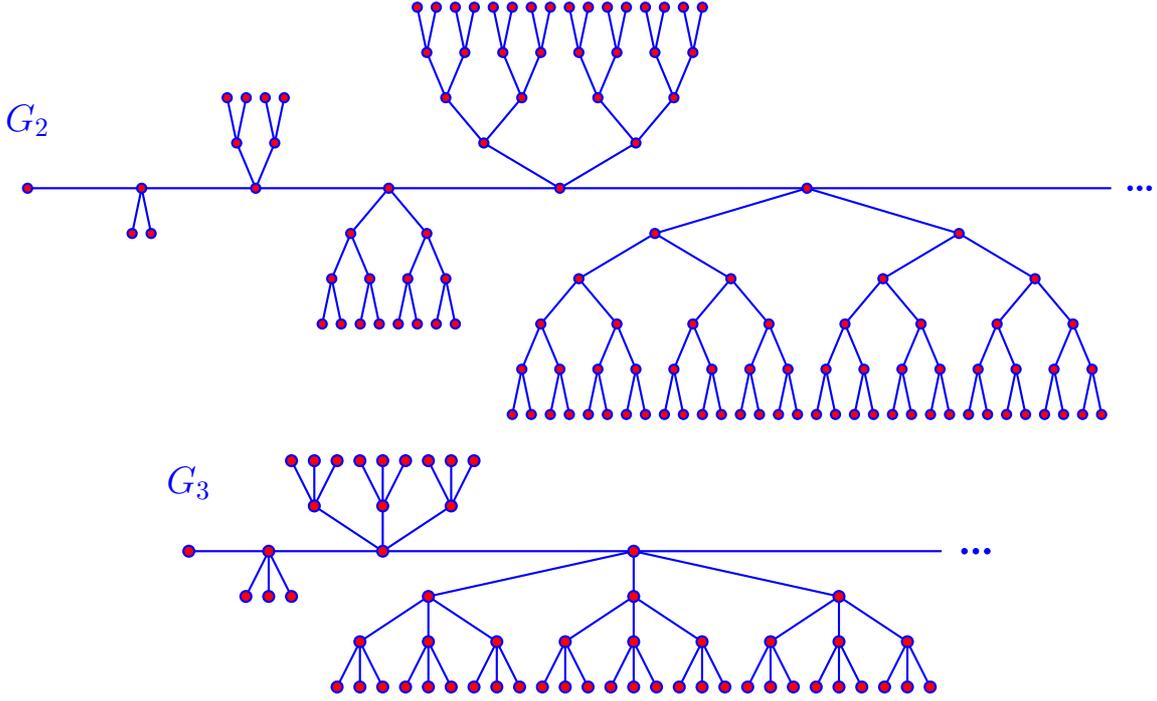

\begin{theorem}
    For $n\geq 2$, $\chi_L(G_n)=\infty$.
\end{theorem}
\begin{proof}
Suppose that $\chi_L(G_n)=t<\infty$ for some $n$. Since $(2i)^t=o\left(n^i\right)$, there is an integer $K$ such that for $i\geq K$, $(2i)^t<n^i$. Note that two leaves in $T(n,i)$ can be only distinguished by the color of a vertex in $T(n,i)$, as they have the same distance to any vertex outside of $T(n,i)$. Since the diameter of $T(n,i)$ is $2i$, then there are at most $(2i)^t$ distinguished leaves in $T(n,i)$. Due to the fact that the number of leaves in $T(n,i)$ is $n^i$ and $(2i)^t<n^i$ for $i\geq K$, there are at least two leaves not distinguished in $T(n,i)$, a contradiction.
\end{proof}

\section{Coloring algorithm}

As a direct consequence of the proof of Theorem \ref{TreeChar}, we have the following corollary.
\begin{corollary}
    Let $T$ be a tree (finite or infinite) with maximum degree $\Delta$, then $\chi_L(T)\leq \chi_L(f(T))+\Delta$.
\end{corollary}

To get a better upper bound for $\chi_L(T)$ for any tree, we combine our coloring algorithm with the coloring algorithm in \cite{DIDP}. Let $T$ be a bounded degree tree with palms $P_1, P_2, \cdots$. Let $u_i$ be the branch vertex in $P_i$. Let $l_i$ be the number of end-paths in $P_i$ and $p_i$ be the number of such end-paths with length one. Let $Q_{i,j}$ be the $j^\text{th}$ end-path in $P_i$.

\begin{theorem}\label{<}
    Let $T$ be a bounded degree tree, then
    $\chi_L(T)\leq \chi_L(f(T))+\max\left\{p_i,\left\lceil\sqrt{l_i}\right\rceil\right\}$.
\end{theorem}
\begin{proof}
    Let $T$ be any tree. Color the tree $T$ using Algorithm $\ref{algo1}$. We will prove that the coloring in Algorithm $\ref{algo1}$ is a locating coloring.
    Note that every vertex in $P_i$ is colored by one of $\{c(u_i),n+1,n+2,\cdots,n+m\}$, so $d_T(v,c_i)=d_{f(T)}(v,c_i)$ for every vertex $v$ in $f(T)$. Any two vertices in $f(T)$ is distinguished by the locating coloring of $f(T)$. A vertex in $f(T)$ and a vertex in $T-f(T)$ are distinguished by colors $\{n+1,n+2,\cdots,n+m\}$.
    
    Suppose there were two vertices $v$ and $w$ with the same color code, then $v$ and $w$ must be in different palms $P_i$ and $P_j$ with $d(v,u_i)=d(w,u_j)$ and $c(u_i)=c(u_j)$. Since $f(T)$ is colored by a locating coloring, then the color that distinguishes $u_i$ and $u_j$ will also distinguish $v$ and $w$, a contradiction.
\end{proof}

\begin{algorithm}
\label{algo1}
\caption{Locating coloring of a tree $T$.}
\begin{algorithmic}
\Procedure{ A locating coloring of a tree $T$}{}
   \State Input: Tree $T$, locating coloring of $f(T)$
   \State Output: A locating coloring $c$ of $T$
   \State Color every vertex in $f(T)$ with the existing locating coloring.
   \State $n \leftarrow \max\{c(v)\mid v\in f(T)\}$
   \State $m \leftarrow \max\{p_i,\sqrt{l_i}\}$
   \State Let $A=\{(a,b) \mid n+1\leq a,b\leq n+m, a\ne b\}$
   \For {$i=1,2,\cdots$}
    \State Sort all end-paths in $P_i$ in ascending order
    \For {$j=1,2,\cdots,l_i$}
     \If {$j\leq m$}
      \State Color $Q_{i,j}$ using $n+j$ and $c(u_i)$ alternately
     \Else 
      \State $k\leftarrow j-m$
      \State Let $(a_k,b_k)$ be the $k^\text{th}$ element in $A$
      \State Color $Q_{i,j}$ using $a_k$ and $b_k$ alternately
     \EndIf
    \EndFor
   \EndFor
\EndProcedure\end{algorithmic}
\end{algorithm}

Although Theorem \ref{<} and Algorithm \ref{algo1} give a good approximation for the locating chromatic number of a tree, it is not practical to use the theorem and the algorithm because we need the locating chromatic number of $f(T)$ and its locating coloring. 
If $T$ is a finite tree, then it is guaranteed that $f^n(T)$ is a path for some $n$. Since the locating chromatic number of a path and its locating coloring is known (Thm \ref{path}), we could apply Theorem \ref{<} and Algorithm \ref{algo1} repeatedly to get a more practical approximation of $\chi_L(T)$ and it's coloring algorithm. This observation also emphasizes the importance of conjecture \ref{TreeChar}, if the conjecture is true, we have a locating coloring of any infinite tree with finite locating chromatic number.

Let $T$ be a bounded degree tree (finite or infinite) such that $f^n(T)$ is a path. For $i=0,1,\cdots,n$, define $T_i:=f^i(T)$, $l^{(i)}_{max}$ the maximum number of end-paths from a palm in $T_i$, and $p^{(i)}_{max}$ the maximum number of end-paths of length one from a palm in $T_i$.

\begin{theorem}\label{<<}
    Let $T$ be a bounded degree tree such that $f^n(T)$ is a path $P$, then
    $$\chi_L(T)\leq \chi_L(P)+\sum_{i=0}^{n-1}\max\left\{p^{(i)}_{max}, \left\lceil\sqrt{l^{(i)}_{max}}\right\rceil\right\}.$$
\end{theorem}
\begin{proof}
    Use Theorem \ref{<} repeatedly.
\end{proof}

We compare our coloring algorithm with some known coloring algorithms, see Table \ref{table}. For larger graphs with many branches, Theorem \ref{<<} works much better than FM\cite{FM19} and BP\cite{DIDP}. For small graphs, Theorem \ref{<<} is not always better than BP\cite{DIDP}. The case of BP\cite{DIDP} is better than Theorem \ref{<<} happens when BP\cite{DIDP} is equal to the exact value, and in this case Theorem \ref{<<} has the same order of locating chromatic number.

\begin{table}[h!]
	\begin{center}
	\caption{Comparison of the locating-chromatic number of trees}
	\label{table}
	\begin{tabular}{|c|c|c|c|c|}
		\hline 
		\multirow{3}{*}{\textbf{Graph} \ $\mathbf{\textit{T}}$} & \multicolumn{4}{c|}{$\mathbf{\chi_L(\textit{T})}$} \\
		\cline{2-5}
		{} & \multirow{2}{*}{\textbf{Exact value}} & \multicolumn{3}{c|}{\textbf{Computational value}} \\
		\cline{3-5}
		{} &{} & \textbf{FM} \cite{FM19} & \textbf{BP} \cite{DIDP} & \textbf{Thm \ref{<<}}  \\
		\hline
		Amalgamation of star & 
		$m+a$  & $km-2k+2$ & $k+m-1$ & $k+m$ \\
		\hline
		Banana tree  & $k-1$   & $n(k-3)+2$ &$n+k-1$ & $\sqrt{n}+k$\\
		\hline
		Caterpillar  & $n+2$   & $nm-m+2$ & $n+m$ & $n+3$\\
		\hline
		Complete $n$-ary tree & $(n+k-1)$ & $n^k-n^{k-1}+2$ & $n^{k-1}+n$ & $nk+1$ \\
		\hline
		Double star  & $b+1$ & $b+a$ & $b+2$  & $b+2$\\ 
		\hline
		Lobster  & $n+2$ & $mn^2-mn+2$ & $n(m+1)$ & $2n+3$ \\
		\hline
		Firecracker-1  & 	$k $  & $n(k-3)+2$ & $n+k-2$ & $k+3$ \\ 
		\hline
		Firecracker-2  & 	$k-1 $  & $n(k-3)+2$ & $n+k-2$ & $k+3$ \\ 
		\hline
		Olive tree & $\left\lceil\log_3\left(\frac{n}{4}\right)\right\rceil+3$ & $k+1$& $\lceil\sqrt{k}\rceil+1$ & $\lceil\sqrt{k}\rceil+1$\\
		\hline
		Star	& $n$  & $n$  & $n$ & $n$ \\ 
		\hline 
	\end{tabular}
\end{center}

\vspace{12pt}
\noindent \textbf{Note}:
\vspace{-6pt}
\begin{center}
	\begin{tabular}{p{5cm}p{7.2cm}}
		$\chi_L(T)$ & Locating-chromatic number of $T$\\
		Amalgamation of star & $S_{k,m}$; $H(a-1)<k\leq H(a), a\geq 1$ \cite{asam}\\
		Banana tree & $B_{n,k}$; $k \geq 4,\ 1<n \leq k-1$ \cite{asban} \\
		Caterpillar & $C(m,n)$; $n,m\geq 3$ \cite{asfi}\\
		Complete $n$-ary tree & $T(n,k)$; $n\geq 2$; $k=2,3$ \cite{welary}\\
		Double star & $S(a,b)$; $1\leq a \leq b$, $b \geq 2$  \cite{ref10}\\
		Lobster & $Lb(m,n)$; $2 \leq m \leq 3(n+2)+1$, $n \geq 2$ \cite{syoflob}\\	   
		Firecracker-1 & $F_{n,k}$; $k\geq 5$, $2 \leq n \leq k-1$  \cite{asfi}\\
		Firecracker-2 & $F_{n,k}$; $k\geq 5$, $ n > k-1$  \cite{asfi}\\
		Olive tree & $O_k$; $k\geq 3$ \cite{YH21} \\
		Star & $S_n$ \cite{ref10} \\
	\end{tabular}
\end{center}

\end{table}

\section*{Acknowledgment}
	This research has been supported by "\textit{PPMI FMIPA ITB}" managed by Faculty of Mathematics and Natural Sciences, Institut Teknologi Bandung.


\begin{thebibliography}{99}
\small 
\baselineskip 12pt

\bibitem{asban} Asmiati, Locating Chromatic Number of Banana Tree, International Mathematical Forum, \textbf{12} (2017), 39 - 45.
		
\bibitem{asam} Asmiati, H. Assiyatun, E.T. Baskoro, Locating-Chromatic number of amalgamation of stars, \textit{ITB J.Sci.} \textbf{43A} (2011), 1--8.

\bibitem{asfi} Asmiati, E.T. Baskoro, H. Assiyatun, D. Suprijanto, R. Simanjuntak, S. Uttunggadewa, The locating-chromatic number of firecracker graphs, Far East J. Math. Sci. \textbf{63} (2012), 11–23.

\bibitem{DIDP} E.T. Baskoro, D.I.D. Primaskun, Improved algorithm for the locating-chromatic number of trees, {\em Theoretical Computer Science}, {\bf 856} (2021), 165--168.

\bibitem{DAM12} J. C\'aceres, C. Hernando, M. Nora, I.M. Pelayo, and M.L. Puertas, On the metric dimensions of infinite graphs, {\em Electron. Notes Discrete Math.}, {\bf 35} (2009) 15--20.

\bibitem{ref10} G. Chartrand, D. Erwin, M.A. Henning, P.J. Slater, P. Zhang, The locating-chromatic number of a graph, \textit{Bull. Inst. Combin. Appl.} \textbf{36} (2002) 89--101.
		
\bibitem{FM19} M. Furuya and N. Matsumoto, Upper bounds on the locating-chromatic number of trees, \textit{Discrete Applied Mathematics}, \textbf{257}(2019), 338-341. 

\bibitem{YH21} Y. Hafidh, E.T. Baskoro, On the locating chromatic number of trees, {\em submitted}.

\bibitem{syoflob} D.K. Syofyan, E.T. Baskoro, H. Assiyatun, On the locating-chromatic number of homogeneous lobsters, \textit{AKCE Int. J. Graphs Comb.} \textbf{10} (2013), 245-–252.

\bibitem{welary}  D. Welyyanti, E.T. Baskoro, R. Simanjuntak, S. Uttunggadewa, On locating-chromatic number of complete n-ary tree, \textit{AKCE Int. J. Graphs Comb. }\textbf{10} (2013), 309-–315.

\end{thebibliography}
\end{document}